\documentclass[reqno]{amsart}

\numberwithin{equation}{section}

\newtheorem{Theorem}{Theorem}[section]
\newtheorem{Definition}[Theorem]{Definition}
\newtheorem{Corollary}[Theorem]{Corollary}
\newtheorem{Lemma}[Theorem]{Lemma}
\newtheorem{Proposition}[Theorem]{Proposition}
\newtheorem{Remark}[Theorem]{Remark}
\newtheorem{Example}[Theorem]{Example}

\newcommand{\<}{\langle}
\renewcommand{\>}{\rangle}
\newcommand{\R}{\mathbb{R}}

\newcommand{\xii}{|\xi|}

\newcommand{\lloc}{\mathrm{\,loc\,}}
\newcommand{\sign}{\mathrm{\,sign\,}}

\title[A test function method for fractional Laplacians]{A test function method for evolution equations with fractional powers of the Laplace operator}
\author{M. D'Abbicco and  K. Fujiwara}

\address{Marcello D'Abbicco, Dept. of Mathematics, University of Bari, Via E. Orabona 4 - 70125 BARI - ITALY}
\address{Kazumasa Fujiwara, Mathematical Institute, Tohoku University, 6-3 Aoba, Aramaki, Aoba-ku, Sendai, Miyagi, 980-8578 Japan}

\begin{document}

\begin{abstract}
In this paper, we discuss a test function method to obtain nonexistence of global-in-time solutions for higher order evolution equations with fractional derivatives and a power nonlinearity, under a sign condition on the initial data. In order to deal with fractional powers of the Laplace operator, we introduce a suitable test function and a suitable class of weak solutions. The optimality of the nonexistence result provided is guaranteed by both scaling arguments and counterexamples. In particular, our manuscript provides the counterpart of nonexistence for several recent results of global existence of small data solutions to the following problem:
\[\begin{cases}
u_{tt} + (-\Delta)^{\theta}u_t + (-\Delta)^{\sigma} u = f(u,u_t), & t>0, \ x\in\R^n,\\
u(0,x)=u_0(x), \ u_t(0,x)=u_1(x),
\end{cases}\]
with~$f=|u|^p$ or~$f=|u_t|^p$, where~$\theta\geq0$ and $\sigma>0$ are fractional powers.
\end{abstract}



\maketitle

%[[[ \section{Introduction}
\section{Introduction}
%[[[ In this paper,
In this paper, we investigate the critical exponent for the nonexistence of global solutions to evolution equations with fractional spatial derivatives and a source power nonlinearity.

We consider a linear differential operator
\begin{equation}\label{eq:L}
L= \partial_t^m + \sum_{j=0}^{m-1} A_j \partial_t^{j},\quad \text{where}\ A_j=a_j(-\Delta)^{\sigma_j},
\end{equation}
with~$m\ge1$, $a_j\in\R$, and~$\sigma_j\geq0$. Here~$(-\Delta)^{\sigma_j}$ represents a (possibly) fractional power of the Laplace operator (see later, Definition~\ref{def:Laplaceop}). We assume that at least one among~$A_j$, with~$j=0,\ldots,m-1$ is a non-integer power of the Laplace operator.

We are interested into find nonexistence results for global-in-time solutions to the Cauchy problem
\begin{equation}
\begin{cases}
Lu = |\partial_t^\ell u|^p, & t>0, \ x\in\R^n,\\
\partial_t^j u(0,x)= u_j(x), & j=0,\ldots,m-1.
\end{cases}
\label{eq:CP}
\end{equation}
In the evolution equation, a source power nonlinearity~$|\partial_t^\ell u|^p$ appears, with~$p>1$, and $\ell$ an integer between $0$ and $m-1$.

Our interest is in a first moment motivated by recent global-in-time existence results derived for evolution equations with structural damping
\begin{equation}
\begin{cases}
u_{tt} + (-\Delta)^{\sigma_1}u_t + (-\Delta)^{\sigma_0} u = f(u,u_t), & t>0, \ x\in\R^n,\\
u(0,x)=u_0(x), \ u_t(0,x)=u_1(x).
\end{cases}
\label{eq:structev}
\end{equation}
In~\cite{DAE17, DAE20}, global existence in time of small data solutions to~\eqref{eq:structev} in low space dimension is proved in the case~$0<\sigma_1<\sigma_0$ for supercritical powers~$p>p_0$ or, respectively, $p>p_1$, when a power nonlinearity~$|f|\approx|u|^p$ or, respectively, $|f|\approx|u_t|^p$ is considered. Here
\[ p_0 = 1 + \frac{2\sigma_0}{n-2\sigma_1},\qquad p_1 = 1+\frac{2\sigma_1}n, \]
if~$2\sigma_1\leq\sigma_0$, whereas
\[ p_0 = 1 + \frac{2\sigma_0}{n-\sigma_0},\qquad p_1 = 1+\frac{\sigma_0}n, \]
if~$2\sigma_1\geq\sigma_0>1$. Also, in~\cite{EL}, global existence in time of small data solutions to
\[
\begin{cases}
u_{tt} + (-\Delta)^{\sigma} u = |u|^p, & t>0, \ x\in\R^n,\\
u(0,x)=u_0(x), \ u_t(0,x)=u_1(x).
\end{cases}
\]
is proved for~$p>1+2\sigma/(n-\sigma)$ when~$\sigma>1$ in low space dimension.

The nonexistence counterpart of these existence results is derived by using a classic test function method,
under the assumption that all powers of the Laplace operator are integer.
That assumption of integer powers is motivated by the inefficacy of the classical test function method,
when dealing with fractional differential operators.
For the details of the classical test function method and related topics,
we refer the reader to~\cite{MP1, MP2, MP, MPhyp} (see also~\cite{DAL03, DAL13,FO16,II19,IO16,LNZ12, LZ19} and references therein). For the application of a modified test function method to Cauchy problems with fractional derivatives in time and classical derivatives in space we address the reader to~\cite{DAcap}. For an application of a test function method to nonlinearities of type~$\mu(|u|)|u|^{p_c}$, where~$\mu$ is a modulus of continuity, we refer to~\cite{EGR}.

In recent times, several authors investigated existence of global-in-time small data solutions for evolution equations with supercritical power nonlinearities, and the importance of having an instrument which provides a counterpart blowup or nonexistence argument for subcritical (and possibly critical) power nonlinearities is crucial. In particular, the counterpart of the global-in-time existence result for the classical damped wave equation~\cite{M76, TY01} ($\sigma_1=0$, $\sigma_0=1$ in~\eqref{eq:structev}) has been derived in the critical case in~\cite{Z01} by a simple application of the test function method (for the wave equation without damping, the methods are different and we address the reader to the classical results in~\cite{Georgiev, GLS, G, G2, J, Kato, Sc, Si, Strauss}, but this list is far from being exhaustive).

In this manuscript, we provide a positive answer to this problem,
constructing a suitable test function which behaves well when fractional Laplace operators are applied.
In particular, we refer to Examples~\ref{ex:Ebert}, \ref{ex:damp} and~\ref{ex:dampt}.
Moreover, with this construction,
we can obtain a nonexistence argument for a wide class of possibly higher order operators, as in~\eqref{eq:L}.

The non-existence of global-in-time weak positive solutions for $u_{t} + (-\Delta)^\sigma u = u^p$, with fractional values of~$\sigma$, has been proved in~\cite{FK10} for~$1<p\leq 1+2\sigma/n$,
using the classical test function method.
This was possible thanks to the pointwise control of fractional derivative derived by C\'ordoba and C\'ordoba
in \cite[Thorem 1]{CC03} and \cite[Proposition 2.3]{CC04},
\begin{equation}\label{eq:CC}
((-\Delta)^{s/2} \phi^2)(x)
\leq 2 \phi(x) (-\Delta)^{s/2} \phi(x)
\end{equation}
for any $x \in \mathbb R^n$, $0 < s < 2$, and $\phi \in \mathcal S$, nonnegative,
where $\mathcal S$ is the Schwartz class.
%The nonnegativity of~$\phi$ is crucial in the previous estimate.
The same approach works for some damped evolution models,
$u_{tt}+2(-\Delta)^{\sigma/2}u_t+(-\Delta)^\sigma u=u^p$, see~\cite{DAR14, DKR}. In these models, a suitable sign assumption on the initial data was sufficient to guarantee the positivity of the solution, which was crucial to effectively employ estimate~\eqref{eq:CC} in the test function method.

On the other hand, positive solutions may not exist in general cases. We mention that estimates of type
\begin{align}
|((-\Delta)^{s/2} \phi^2)(x)|
\leq 2 | \phi(x) (-\Delta)^{s/2} \phi(x) |
\label{eq:PC}
\end{align}
for any $x \in \mathbb R^n$ with some $s > 0$ and $\phi \in C^\infty$ do not hold generally with compactly supported functions
because $(-\Delta)^{s/2}$ is non-local when $s$ is not even number.
In \cite{F18},
it is shown that
\eqref{eq:PC} holds for $0 < s < 2$ and $\phi(x) = \langle x \rangle^{-q}$
with some $q > 0$.
%In this manuscript, we generalize the estimate of \cite{F18} and show the non-existence of time-global weak solutions for \eqref{eq:CP}.

The second goal of this manuscript is to provide a constructive method which allows us to directly compute the critical exponents of nonexistence for operators as in~\eqref{eq:L}, simply knowing the fractional powers~$\sigma_j$ appearing in~\eqref{eq:L}. Namely, we find the best possible scaling which relates the time and space variable, for which a nonexistence result holds for~$Lu=|\partial_t^\ell u|^p$.

For the ease of reading, the paper is organized in sections. In each section, we discuss an aspect of the problem considered:
\begin{itemize}
\item in Section~\ref{sec:weaksol}, we give a suitable definition of test function, which fits our need to behave well with fractional Laplace operators, and consequently a definition of weak solution to~\eqref{eq:CP};
\item in Section~\ref{sec:fractLap}, we discuss the action of the fractional Laplacian operators on the test function considered;
\item in Section~\ref{sec:critical}, we discuss how the critical exponent for~\eqref{eq:CP} is obtained for different operators by a constructive method;
\item in Section~\ref{sec:proof}, we give the proof of Theorem~\ref{thm:main};% and~\ref{thm:main2};
\item in Section~\ref{sec:examples}, we provide examples of the application of Theorem~\ref{thm:main}% and~\ref{thm:main2}
to some equations;
\item in Section~\ref{sec:integer}, we briefly discuss the case of an equation with classical derivatives, for the ease of reference, showing the validity of the critical exponent constructed in Section~\ref{sec:critical}.
\end{itemize}

Having in mind that the questions and details about ``what is a weak solution'', ``how the fractional Laplacian operator act on the test function'', and ``how the critical exponent is derived'' are postponed, respectively, to Sections~\ref{sec:weaksol}, \ref{sec:fractLap} and~\ref{sec:critical}, we are ready to state our nonexistence results.

First of all, we construct the critical exponent for~\eqref{eq:CP}.
\begin{Definition}\label{def:homdim}
Let~$L$ be as in~\eqref{eq:L}. To uniform the notation, we denote~$a_m=1$ and~$\sigma_m=0$ (here~$(-\Delta)^0=\mathrm{Id}$, the identity operator), consistently with the notation
\[ L= \sum_{j=0}^{m} A_j \partial_t^{j},\quad \text{where}\ A_j=a_j(-\Delta)^{\sigma_j}. \]
For any~$\eta\in[0,\infty]$, we define the function
\[
g(\eta)= \min_{j=0,\ldots,m, a_j \neq 0} \{ (j-\ell) \eta + 2\sigma_j \}.
\]
Then we define
\begin{equation}\label{eq:pcritical}
p_c=\max_{\eta \in[0,\infty]} h(\eta), \ \text{where}\quad
	h(\eta)
	= \frac{n+\eta}{(n+\eta-g(\eta))_+}
	= 1 + \frac{g(\eta)}{(n+\eta-g(\eta))_+},
\end{equation}
where $b_+ = \max(b,0)$ for $b\in \mathbb R$ and we set $1/0=\infty$. We call~$p_c$ as in~\eqref{eq:pcritical} \emph{the critical exponent} for~\eqref{eq:CP}.
\end{Definition}
%]]]

%[[[ \begin{Definition}\label{def:critical exponent}
%]]]

%[[[ \begin{Theorem}\label{thm:main}
Our main result is the following.

\begin{Theorem}\label{thm:main}
Let~$L$ be as in~\eqref{eq:L}, in particular at least one among~$A_j$ is a non-integer power of the Laplace operator. We set~$p_c$ be as in Definition~\ref{def:homdim} and we fix~$q=n+2s$, where~$s\in(0,1)$ is given by
\begin{equation}\label{eq:s}
s = \min \{ \sigma_j-[\sigma_j]: \ \text{$\sigma_j$ is not integer and~$a_j\neq0$} \}.
\end{equation}
By~$[\sigma]$ we denote the largest integer, smaller than or equal to~$\sigma$, i.e., its floor function. We define
\[ I = \{ j \geq \ell: \ \sigma_{j+1}=0,\ a_{j+1} \neq 0\}. \]
We assume that~$u_j=0$ for any~$j\leq \ell-1$, that~$u_j\in L^1(\<x\>^qdx)$, for any~$j=\ell,\ldots,m-1$, with~$j\not\in I$, and that~$u_j\in L^1$ for any~$j\in I$. Here and in the following, we use the notation
\[ \<x\>=(1+|x|^2)^{\frac12}.\]
Moreover, we assume the sign condition
\begin{equation}\label{eq:sign}
\sum_{j \in I} a_{j+1} \int_{\R^n} u_j(x)\,dx > 0,
\end{equation}
where~$a_j$ are as in~\eqref{eq:L}. If there exists a global-in-time weak solution
\[ u \in W^{\ell,p}_\lloc([0,\infty),L^p(\R^n,\<x\>^{q}dx))\]
to~\eqref{eq:CP}, according to Definition~\ref{def:weak}, then~$p>p_c$.% $p \geq p_c$. Moreover, if~$p=p_c$, then~$\partial_t^\ell u \in L^{p_c}([0,\infty)\times\R^n)$.
\end{Theorem}
%]]]

%[[[ \begin{Remark}
\begin{Remark}
We stated Theorem~\ref{thm:main} with the most possible general definition of weak solution, which was consistent with the special test function employed in our argument (our main difficulty was in dealing with a non compactly supported test function). In particular, the nonexistence result applies to more regular solutions. In other words, regular solutions are also weak solutions, as it is customary. We show this in Section~\ref{sec:weaksol}.
\end{Remark}
%]]]

The sharpness of the critical exponent~$p_c$ is discussed in two ways: by a scaling argument in Section~\ref{sec:critical}, and by concrete examples collected in Section~\ref{sec:examples}. For these models, we may prove a global existence result for supercritical powers~$p>p_c$, and this shows the sharpness of the critical exponent~$p_c$ found in our paper.
%]]]

%[[[ \section{Test function and weak solutions}\label{sec:weaksol}
\section{Test function and weak solutions}\label{sec:weaksol}

%[[[ In this paper we employ a modified test function method to derive our results.
In this paper we employ a modified test function method to derive our results.
In order to deal with nonlocal operators, as the fractional Laplace operator is,
we will replace compactly supported test functions by suitable test functions with a polynomial decay.
To present this approach, we introduce a definition of weak solution to problem~\eqref{eq:CP} which fits our scopes.
%]]]

%[[[ \begin{Definition}\label{def:class}
\begin{Definition}\label{def:class}
Let~$L$ be as in~\eqref{eq:L} and fix~$q=n+2s$, where~$s\in(0,1)$ as in~\eqref{eq:s}. We define the space~$\mathcal C_q^\infty(\R^n)$ as the subspace of infinitely differentiable functions~$\varphi$ such that~$\<x\>^q\varphi$ is bounded, and for any~$\sigma>0$, with~$\sigma$ integer or~$\sigma-[\sigma]\in[s,1)$, the function $\<x\>^q (-\Delta)^\sigma\varphi$ is bounded.
\end{Definition}
\begin{Remark}
The space~$\mathcal C_q^\infty$ is a vector space. We will show in Section~\ref{sec:fractLap} that it is nonempty, in particular, the function~$\varphi(x)=\<x\>^{-r}$ is in~$\mathcal C_q^\infty$, for any~$r\geq q$. We notice that, due to~$q>n$, we get the inclusion
\[ \mathcal C_q^\infty\subset L^\infty(\R^n,\<x\>^{q}dx)\subset L^1.\]
Moreover, if~$\varphi\in \mathcal C_q^\infty$, then~$A_j\varphi\in L^\infty(\R^n,\<x\>^{q}dx)\subset L^1$ for any~$j=0,\ldots,m$, as well, as a consequence of~\eqref{eq:s}.
\end{Remark}
%]]]

%[[[ \begin{Definition}\label{def:weak}
\begin{Definition}\label{def:weak}
Let~$L$ be as in~\eqref{eq:L} and fix~$q=n+2s$, where~$s\in(0,1)$ as in~\eqref{eq:s}. Assume that the initial data in~\eqref{eq:CP} verify the assumption
\[ \text{$u_j=0$ if~$j=0,\ldots,\ell-1$ and~$u_j\in L^1(\<x\>^{-q}dx)$ if $j \geq \ell$.}\]
We fix~$T\in(0,\infty]$. We say that~$u \in W^{\ell,p}_\lloc([0,T),L^p(\R^n,\<x\>^{-q}dx))$ is a weak solution to~\eqref{eq:CP}
if~$\partial_t^ju(0,\cdot)=0$ for any~$j\leq\ell-1$, and for any function~$\psi\in\mathcal C_c^\infty([0,T))$, with~$\psi=1$ in a neighborhood of~$0$ and for any function~$\varphi\in \mathcal C_q^\infty(\R^n)$, it holds
	\begin{align}
	&\int_0^T \psi(t)\,\int_{\R^n}|\partial_t^\ell u(t,x)|^p\,\varphi(x)\,dx\,dt
	\nonumber\\
	& \qquad = \sum_{j=0}^{m} (-1)^{(j-\ell)} \int_0^T \psi^{(j-\ell)}(t)\,\int_{\R^n} \partial_t^\ell u (t,x)\,A_j\varphi(x)\,dx\,dt
	\nonumber\\
	& \qquad \qquad - \sum_{j=\ell}^{m-1} \int_{\R^n}u_{j}(x)\,A_{j+1}\varphi(x)\,dx,
	\label{eq:1.4}
	\end{align}
where for $j < 0$, $\psi^{(j)}$ is the compactly supported primitive of~$\psi^{(j+1)}$, %inductively given by
	\[
	\psi^{(j)}(t)
	= - \int_t^T \psi^{(j+1)}(\tau) d\tau.
	\]
We say that the weak solution is locally-in-time defined if~$T<\infty$, and is globally-in-time defined if~$T=\infty$. Equivalently, a function $u \in W^{\ell,p}_\lloc([0,\infty),L^p(\R^n,\<x\>^{-q}dx))$ is a global-in-time weak solution if, and only if, $u|_{[0,T)\times\R^n}$ is a local-in-time weak solution, for any~$T>0$.
\end{Definition}
\begin{Remark}\label{rem:Aphi}
We recall that in Definition~\ref{def:weak} and in the following, by~$\mathcal C_c^\infty([0,T))$, we denote the infinitely differentiable functions in~$[0,T)$, with compact support. In particular, functions in~$\mathcal C_c^\infty([0,T))$ vanish as~$t\to T$, but may assume nonzero values at~$t=0$, since the interval~$[0,T)$ is left-closed.

We stress that we fix a test function~$\psi$ with a constant value in a neighborhood of the origin in order to simplify the definition of weak solution, see the proof of Proposition~\ref{prop:classical}. The choice of assuming null initial data~$u_j$ in~\eqref{eq:CP}, for~$j\leq \ell-1$ has the same motivation. A definition of weak solution working for generic test functions~$\psi\in\mathcal C_c^\infty([0,T))$ may be easily derived following the proof of Proposition~\ref{prop:classical}.
\end{Remark}

%]]]

%[[[ \begin{Remark}
\begin{Remark}
We remark that when~$\ell=0$, it holds
\[ L^p_\lloc([0,T),L^p(\R^n,\<x\>^{-q}dx)) \subset L^p_\lloc([0,T)\times\R^n) \]
so that the weak solution space in Definition~\ref{def:weak} is properly contained in a more customary weak solution space used when one deals with ``classical'' test functions~$\varphi\in\mathcal C_c^\infty$ and classical derivatives, say~$\sigma_j$ are integers (see Section~\ref{sec:integer}). We recall that~$u \in L^p_\lloc([0,T),X)$ where~$X$ is a normed functional space, if for any~$T_1\in[0,T)$, $u|_{[0,T_1]} \in L^p([0,T_1],X)$, that is,
\[ \int_0^{T_1} \|u(t,\cdot)\|_X^p\,dt <\infty. \]
We emphasize the crucial difference between the space~$L^p_\lloc((0,T),X)$ and its proper subspace~$L^p_\lloc([0,T),X)$.

The choice of the solution space $L^p_\lloc([0,T),L^p(\R^n,\<x\>^{-q}dx))$
is motivated by the fact that we are dealing with test functions of type~$\varphi(x)=\<x\>^{-q}$ in this paper, where~$q>n$.

When~$\ell\geq1$, the weak solution space is modified in $W^{\ell,p}_\lloc([0,T),L^p(\R^n,\<x\>^{-q}dx))$, so that the weak derivative~$\partial_t^\ell u$ exists in $L^p_\lloc([0,T),L^p(\R^n,\<x\>^{-q}dx))$. Moreover, by standard embeddings, $u\in \mathcal C^{\ell-1}([0,T),L^p(\R^n,\<x\>^{-q}dx))$, so that the initial condition~$\partial_t^ju(0,\cdot)=0$ (a.e.) for any~$j\leq\ell-1$ is well-defined.
\end{Remark}
%]]]

%[[[ \begin{Remark}
\begin{Remark}
We notice that the integrals in Definition~\ref{def:weak} are well-defined for weak solutions. In particular, due to $A_j\varphi\in L^\infty(\<x\>^qdx)$ for any~$j=0,\ldots,m$ (Remark~\ref{rem:Aphi}), the integral
\[
\int_{\R^n}u_{j}(x)\,A_{j+1}\varphi(x)\,dx
\]
is well-defined, for any~$j=\ell,\ldots,m-1$, and
\begin{align*}
& \int_{\R^n}|\partial_t^\ell u(t,x)|\,|A_j\varphi(x)|\,dx \\
& \qquad \leq \Big( \int_{\R^n}|\partial_t^\ell u(t,x)|^p\,\<x\>^{-q}\,dx\Big)^{\frac1p}\,\Big(\int_{\R^n}\<x\>^{qp'/p}\,|A_j\varphi(x)|^{p'}\,dx\Big)^{\frac{1}{p'}},
\end{align*}
with the latter integral being bounded by
\[ \int_{\R^n}\<x\>^{qp'/p}\,|A_j\varphi(x)|^{p'}\,dx \leq C\,\int_{\R^n}\<x\>^{-q}\,dx\leq C', \]
due to~$q>n$.
\end{Remark}

%]]]

We may easily show that smooth, classical solutions are weak solutions.

%[[[ \begin{Proposition}\label{prop:classical}
\begin{Proposition}\label{prop:classical}
Assume that~$u_j=0$ for any $j = 0, \cdots, \ell - 1$, and that~$u_j\in\mathcal S$, for any~$j=\ell, \ldots,m-1$ in~\eqref{eq:CP},
where~$\mathcal S(\R^n)$ is the Schwartz space. Assume that~$u\in\mathcal C^m([0,T),\mathcal S)$ is a ``classical'' solution to~\eqref{eq:CP}.
Then, $u$ is also a weak solution to~\eqref{eq:CP}, according to Definition~\ref{def:weak}.
\end{Proposition}
\begin{proof}
We multiply the equation~$Lu=|\partial_t^\ell u|^p$ in~\eqref{eq:CP}, by~$\psi(t)\varphi(x)$, and we integrate with respect to time and space. Due to~$\partial_t^ju(t,\cdot)\in\mathcal S$ and~$\varphi\in\mathcal C_q^\infty$, by integration by parts in space, we first get
\[
\int_0^T \psi(t)\,\int_{\R^n}\varphi(x)\partial_t^j A_j\,u(t,x)\,dx\,dt = \int_0^T \psi(t)\,\int_{\R^n}\partial_t^j u(t,x)\,A_j\varphi(x)\,dx\,dt,
\]
for any~$j=0,\ldots,m$. For any~$j=\ell+1,\ldots,m$, after~$j-\ell$ steps of integration by parts in time, recalling that~$\psi=1$ in a neighborhood of~$t=0$ and is compactly supported in~$[0,T)$, we then obtain
\begin{align*}
&\int_0^T \psi(t)\,\int_{\R^n}\partial_t^j u(t,x)\,A_j\varphi(x)\,dx\,dt\\
& \qquad = -\,\int_{\R^n}\partial_t^{j-1} u(0,x)\,A_j\varphi(x)\,dx\\
& \qquad \qquad +(-1)^{j-\ell} \int_0^T \psi^{(j-\ell)}(t)\,\int_{\R^n} \partial_t^{\ell} u(t,x)\,A_j\varphi(x)\,dx\,dt\\
& \qquad = -\,\int_{\R^n} u_{j-1}(x)\,A_j\varphi(x)\,dx\\
& \qquad \qquad +(-1)^{j-\ell} \int_0^T \psi^{(j-\ell)}(t)\,\int_{\R^n} \partial_t^{\ell} u(t,x)\,A_j\varphi(x)\,dx\,dt,
\end{align*}
where in the last equality we replaced the initial condition~$\partial_t^{j-1} u(0,x)=u_{j-1}(x)$.

If~$\ell\geq1$, we shall also integrate by parts the term in which the $j$-th time derivative of~$u$ appears, for any~$j=0,\ldots,\ell-1$. After the first step of integration by parts, we obtain
\begin{align*}
& \int_0^T \psi(t)\,\int_{\R^n} \partial_t^j u(t,x)\,A_j\varphi(x)\,dx\,dt\\
& \qquad = - \psi^{(-1)}(0)\, \int_{\mathbb R^n} \partial_t^j u(0,x)\,A_j\varphi(x) dx \\
& \qquad \qquad - \int_0^T \psi^{(-1)}(t) \int_{\mathbb R^n} \partial_t^{j+1} u(t,x)\,A_j\varphi(x) dx \, dt\\
& \qquad = - \int_0^T \psi^{(-1)}(t) \int_{\mathbb R^n} \partial_t^{j+1} u(t,x)\,A_j\varphi(x) dx \, dt,
\end{align*}
where in the last equality we replaced the initial condition~$\partial_t^j u(0,x)=0$. Similarly, after a total of~$\ell-j$ steps of integration by parts, we get
\begin{align*}
& \int_0^T \psi(t)\,\int_{\R^n} \partial_t^j u(t,x)\,A_j\varphi(x)\,dx\,dt\\
& \qquad = (-1)^{\ell-j} \int_0^T \psi^{(j-\ell)}(t) \int_{\mathbb R^n} \partial_t^{\ell} u(t,x)\,A_j\varphi(x) dx \, dt.
\end{align*}
This concludes the proof.
\end{proof}
%]]]

%]]]

%[[[ \section{Fractional Laplacian and its action on the test function}\label{sec:fractLap}
\section{Fractional Laplacian and its action on the test function}\label{sec:fractLap}

There are several possible definitions of fractional powers of the Laplace operator, which are equivalent on suitable classes of functions.

%[[[ \begin{Definition}\label{def:Laplaceop}
\begin{Definition}\label{def:Laplaceop}
For any~$s>0$, we may define the fractional Laplace operator~$(-\Delta)^{\frac{s}2}: H^s\to L^2$, as
\[ (-\Delta)^{\frac{s}2} f = \mathfrak{F}^{-1}(\xii^s\hat f), \]
where $\mathfrak F$ is the Fourier transformation in~$L^2(\R^n)$, and we denote~$\hat f=\mathfrak{F}(f)$.
When~$s$ is an even integer, the definition is consistent with the definition of integer power of the Laplace operator.

If~$s>0$ is not an even integer, the operator~$(-\Delta)^{\frac s 2}$ admits an integral representation.
For $y \in \mathbb R^n$, let $\tau_y$ be the translation operator given by $\tau_y f(x) = f(x+y)$. Then the identity
\begin{align}
(-\Delta)^{s/2} f(x)
= (-1)^{[s/2]+1} C_s \int_{\mathbb R^n} \frac{(\tau_{y/2} - \tau_{-y/2})^{2[s/2]+2} f(x)}{|y|^{n+s}} dy
\label{eq:intrep}
\end{align}
holds for any~$f\in H^s$, where
\[
C_s
= 2^{-2[s/2]-2+s} \bigg( \int_{\mathbb R^n} \frac{\sin(y_1)^{2[s/2]+2}}{|y|^{n+s}} dy \bigg)^{-1}
> 0.
\]
The identity above is essentially shown in \cite{NPV2012}.
%In particular, $(-\Delta)^{\frac{s}2}$ acts on the Schwartz space~$\mathcal S$ and so its action may be extended by duality to tempered distributions.
\end{Definition}
%]]]

The fractional Laplace operator may be conveniently extended to more general spaces, in particular, it may be extended by duality to the tempered distribution space.

We collect some pointwise controls for the fractional derivative of bounded functions with bounded derivatives. For these functions, the action of the fractional Laplace operator may be defined via~\eqref{eq:intrep}.
%
%[[[ Lemma : Fractional Derivative
\begin{Lemma}
\label{Lemma:2.1}
%The embedding
%	\begin{align}
%	(-\Delta)^{-\sigma} L^\infty \hookrightarrow L^\infty \cap (-\Delta)^{-1} L^\infty
%	\label{eq:2.1}
%	\end{align}
%holds for $0 < \sigma < 1$.
Assume $f \in C^2$, bounded, with bounded derivatives. If there exists a constant $C_0$ such that the estimate
	\[
	|f(y)| \leq C_0 |f(x)|,\quad
	\sup_{|\alpha|=2} |\partial^\alpha f(y)|
	\leq C_0 \sup_{|\alpha|=2} |\partial^\alpha f(x)|
	\]
hold when $|x| \leq |y|$,
then for $|x| > 1$,
the following pointwise estimate holds:
	\begin{align}
	| (-\Delta)^{\sigma} f (x)|
	&\leq C {|x|^{-n-2 \sigma}} \int_{|y| < 3 |x|} |f(y)| dy
	+ C |f(x)| |x|^{-2 \sigma}
	\nonumber\\
	&+ \frac{2^{3-2 \sigma}}{2-2 \sigma} C
	|x|^{2-2 \sigma}
	\sum_{|\alpha|=2} \frac{|\alpha|}{\alpha!}
	|\partial^\alpha f ( \frac x 2 ) |,
	\label{eq:2.2}
	\end{align}
for any~$\sigma\in(0,1)$, where the action of the fractional Laplace operator is defined by~\eqref{eq:intrep}.
\end{Lemma}
%]]]

%[[[ Proof part
\begin{proof}
%[[[ The Gagliardo-Nirenberg inequality implies that \eqref{eq:2.1} holds.
%\eqref{eq:2.1} is a direct consequence of the Gagliardo-Nirenberg inequality.
%]]]

%[[[ We modify \eqref{eq:intrep} so as to avoid from the principle value.
%Next, we show \eqref{eq:2.2}.
We conveniently modify \eqref{eq:intrep}, i.e.,
\[ (-\Delta)^{\sigma} f(x) = -C_{2\sigma} \int_{\mathbb R^n} \frac{f(x+y)-2f(x)+f(x-y)}{|y|^{n+2\sigma}} dy. \]
The Taylor theorem implies that
	\begin{align*}
	f(x+y) - f(x)
	&= \nabla f(x) \cdot y
	+ \sum_{|\alpha|=2} \frac{|\alpha|}{\alpha!} y^\alpha \int_0^1 (1-\theta) \partial^\alpha f(x + \theta y ) d \theta,\\
	f(x-y) - f(x)
	&= - \nabla f(x) \cdot y
	+ \sum_{|\alpha|=2} \frac{|\alpha|}{\alpha!} (-y)^\alpha \int_0^1 (1-\theta) \partial^\alpha f(x - \theta y ) d \theta.
	\end{align*}
Thereofore, by the symmetry,
$(-\Delta)^{\sigma} f$ is expressed by
	\begin{align}
	(-\Delta)^{\sigma} f(x)
	&= 2 C_{2 \sigma} \thinspace
	\int_{|y|>r} \frac{f(x) - f(x+y)}{|y|^{n+2 \sigma}} dy
	\nonumber\\
	& - 2 C_{2 \sigma} \thinspace
	\sum_{|\alpha|=2} \frac{|\alpha|}{\alpha!}
	\int_{|y|<r} \frac{y^\alpha}{|y|^{n+2\sigma}}
	\int_0^1 (1-\theta) \partial^\alpha f(x + 2 \theta y ) d \theta \thinspace dy
	\label{eq:2.3}
	\end{align}
with any $r > 0$.
We put $r = |x|/2$.
We estimate the first term of the right-hand side of \eqref{eq:2.3}.
At first we have
	\begin{align}
	\int_{|y|>|x|/2} \frac{f(x)}{|y|^{n+2 \sigma}}d y
	&= C f(x) | x |^{-2 \sigma}.
	\label{eq:2.4}
	\end{align}
The following estimates also hold:
	\begin{align}
	&\int_{|y|> |x|/2} \frac{|f(x + y)|}{|y|^{n+2 \sigma}}d y
	\nonumber\\
	&\leq \int_{|x|/2 < |y| < 2 |x|} \frac{|f(x + y)|}{|y|^{n+2 \sigma}}d y
	+ \int_{|y| > 2 |x|} \frac{|f(x + y)|}{|y|^{n+2 \sigma}}d y
	\nonumber\\
	&\leq C {|x|^{-n-2 \sigma}} \int_{|y| < 3 |x|} |f(y)| dy
	+ C C_0 |f(x)| |x|^{- 2 \sigma}.
	\label{eq:2.5}
	\end{align}
The second term of the right-hand side of \eqref{eq:2.3} is estimated by
	\begin{align}
	& \bigg|
	\sum_{|\alpha|=2} \frac{|\alpha|}{\alpha!}
	\int_{|y|<|x|/2} \frac{y^\alpha}{|y|^{n+ 2 \sigma}}
	\int_0^1 (1-\theta) \partial^\alpha f(x + \theta y ) d \theta \thinspace dy \bigg|
	\nonumber\\
	&\leq C
	\sum_{|\alpha|=2} \frac{|\alpha|}{\alpha!}
	\int_{|y|< |x|/2} \frac{dy}{|y|^{n+2 \sigma-2}}
	|\partial^\alpha f ( \frac x 2 ) |
	\nonumber\\
	&\leq \frac{2^{2-2 \sigma}}{2- 2 \sigma} C
	|x|^{2-2\sigma}
	\sum_{|\alpha|=2} \frac{|\alpha|}{\alpha!}
	|\partial^\alpha f ( \frac x 2 ) |.
	\label{eq:2.6}
	\end{align}
Since \eqref{eq:2.4}, \eqref{eq:2.5}, and \eqref{eq:2.6} imply \eqref{eq:2.2}, we conclude the proof.
%]]]
\end{proof}
%]]]

As a consequence of Lemma~\ref{Lemma:2.1}, we derive the following.

%[[[ \begin{Corollary}\label{cor:test}
\begin{Corollary}\label{cor:test}
Let~$f(x)=\<x\>^{-q}$, for~$q>n$, and let~$\sigma>0$.
We set~$s=\sigma-[\sigma]$. Then
\begin{align}
\forall x\in\R^n:\quad |(-\Delta)^\sigma f(x)|\leq C\,\<x\>^{-q_\sigma}
\label{eq:pointwise control}
\end{align}
where $q_\sigma=q+2\sigma$ if~$\sigma$ is an integer, or~$q_\sigma=n+2s$ otherwise,
and the constant~$C$ verifies the following bound from below:
\begin{align}
C=C(n,\sigma,q)
\geq (-\Delta)^\sigma f(0) = 2^{2\sigma}\,\frac{\Gamma(\sigma+n/2)}{\Gamma(n/2)}\,\frac{\Gamma(\sigma+q/2)}{\Gamma(q/2)}.
\label{eq:C}
\end{align}
Here, $\Gamma$ denotes the gamma function.
\end{Corollary}
\begin{proof}
It is easy to check that
\begin{equation}\label{eq:iteration}
(-\Delta)^{[\sigma]} f (x) = \sum_{k=0}^{[\sigma]} c_k \, \<x\>^{-q-2[\sigma]-2k},
\end{equation}
for some~$c_k=c_k(n,\sigma,q)\in\R$.
Indeed, it is sufficient to iterate the equality
\begin{align*}
-\Delta f(x)
    & = q \sum_{j=1}^n \partial_{x_j} \big( x_j \<x\>^{-q-2} \big) \\
    & = q \sum_{j=1}^n \big( \<x\>^{-q-2} -(q+2) x_j^2 \<x\>^{-q-4} \big)\\
    & = -q (q+2-n) \<x\>^{-q-2} + q(q+2)\<x\>^{-q-4}.
\end{align*}
Therefore, \eqref{eq:pointwise control} holds when~$\sigma>0$ is an integer, that is, $s=0$.
Let us assume that~$\sigma$ is not an integer, that is, $s\in(0,1)$.
Applying Lemma~\ref{Lemma:2.1} separately to each term~$\<x\>^{-q-2[\sigma]-2k}$ in~\eqref{eq:iteration}, we may estimate
\[ |(-\Delta)^s\<x\>^{-q-2[\sigma]-2k}| \leq C\,\<x\>^{-n-2s}, \]
and this implies \eqref{eq:pointwise control}.

Estimate~\eqref{eq:C} follows from the identity
	\begin{equation}\label{eq:der0}
	(-\Delta)^{\sigma} [ \langle \cdot \rangle^{-q} ] (0)
	= 2^{2\sigma}\,\frac{\Gamma(\sigma+n/2)}{\Gamma(n/2)}\,\frac{\Gamma(\sigma+q/2)}{\Gamma(q/2)}.
	\end{equation}
Indeed, for~$-1<\nu<2\mu+3/2$, it holds (see~\cite[p.50]{S}, see also~\cite[\textsection 13.6, p.434]{TB})
\[ \int_0^\infty \frac{\rho^{\nu+1}}{(\rho^2+1)^{\mu+1}}\,J_\nu(r\rho)\,d\rho = (r/2)^\mu \frac{K_{\nu-\mu}(r)}{\Gamma(\mu+1)}, \]
where~$J_\nu$ and~$K_\nu$ are the Bessel functions of, respectively, first kind and second kind, with order~$\nu$. Then, setting
\[ \nu=n/2-1, \quad \mu=q/2-1, \]
and~$r=|\xi|$, we get
\begin{align*}
\mathfrak F [ \langle \cdot \rangle^{-q} ] (\xi)
    & = \frac{1}{(2 \pi)^{n/2}} \int_{\R^n} \frac{e^{-ix\cdot\xi}}{(1+|x|^2)^{q/2}} dx
    %& = \int_0^\infty \frac{\rho^{n-1}\,\tilde J_{\frac{n}2-1}(\rho|\xi|)}{(1+\rho^2)^{q/2}} d\rho\\
    = |\xi|^{1-\frac{n}2}\, \int_0^\infty \frac{\rho^{\frac{n}2}\,J_{\frac{n}2-1}(\rho|\xi|)}{(1+\rho^2)^{q/2}} d\rho\\
    & = |\xi|^{1-\frac{n}2}\,(|\xi|/2)^{q/2-1}\,\frac{K_{(n-q)/2}(|\xi|)}{\Gamma(q/2)}
    = |\xi|^{\frac{q-n}2}\,2^{1-q/2}\,\frac{K_{(n-q)/2}(|\xi|)}{\Gamma(q/2)}.
\end{align*}
On the other hand, for~$\mu>|\nu|$, it holds (see~\cite[\textsection 13.21, p.388]{TB})
\[ \int_0^\infty K_\nu(r)\,r^{\mu-1}\,dr = 2^{\mu-2}\,\Gamma((\mu-\nu)/2)\,\Gamma((\mu+\nu)/2), \]
so that
\begin{align*}
(-\Delta)^{\sigma} [ \langle \cdot \rangle^{-q} ] (0)
	&= \frac{1}{(2 \pi)^{n/2}} \int_{\mathbb{R}^n} |\xi|^{2\sigma} \mathfrak F [ \langle \cdot \rangle^{-q} ] (\xi) d \xi\\
    &= \frac{|S^{n-1}|}{(2 \pi)^{n/2}}\,\frac{2^{1-q/2}}{\Gamma(q/2)}\,\int_0^\infty r^{2\sigma+n-1+\frac{q-n}2}\,K_{(n-q)/2}(r)\,dr \\
    &= \frac{2\pi^{n/2}}{(2 \pi)^{n/2}\Gamma(n/2)}\,\frac{2^{1-q/2}}{\Gamma(q/2)}\,2^{2\sigma+n+\frac{q-n}2-2}\,\Gamma(\sigma+q/2)\,\Gamma(\sigma+n/2)\\
    &= 2^{2\sigma}\,\frac{\Gamma(\sigma+n/2)}{\Gamma(n/2)}\,\frac{\Gamma(\sigma+q/2)}{\Gamma(q/2)},
\end{align*}
where we put
\[ \nu=\frac{n-q}2, \quad \mu=2\sigma+n+\frac{q-n}2, \]
and we replaced~$|S^{n-1}|=2\pi^{n/2}/\Gamma(n/2)$.
\end{proof}
%]]]

Thanks to Corollary~\ref{cor:test}, we get that the function~$\phi(x)=\<x\>^{-q}$, where~$q=n+2s$,
with~$s$ defined as in~\eqref{eq:s}, belongs to the test functions space~$\mathcal C_q^\infty$,
introduced in Definition~\ref{def:class}.

%]]]

%[[[ \section{The critical exponent}\label{sec:critical}
\section{The critical exponent}\label{sec:critical}

%[[[ In this section, we describe how we derived the critical exponent in Definition~\ref{def:homdim},
In this section, we describe how we derived the critical exponent in Definition~\ref{def:homdim},
and how to compute it.
The function~$g$ is designed in order to describe the scaling property of the operator~$L$ defined in~\eqref{eq:L},
according to a scaling parameter~$\eta$.
Namely, if one applies a scaling argument to the time and space variable of the problem~\eqref{eq:CP} by replacing $(t,x)\mapsto (R^{-\eta},R^{-1}x)$ for some~$R\gg1$, then
\[ \partial_t^{j-\ell}(\psi(R^{-\eta}t))\,A_j(\varphi(R^{-1}x)) = R^{-(j-\ell)\eta-2\sigma_j}\,(\partial_t^{j-\ell}\psi)(R^{-\eta}t)\,(A_j\varphi)(R^{-1}x).  \]
The terms above are the ones appearing in the definition of weak solution given by~\eqref{eq:1.4}, and are related to the scaling of each term in~$L$.

Since we are interested in taking the limit as~$R\to\infty$, for any possible scaling parameter~$\eta$, the function~$g(\eta)$ gives the largest quantity of type~$R^{-(j-\ell)\eta-2\sigma_j}$, for any~$j$. Namely,
\[ \forall R\geq1, \qquad R^{-(j-\ell)\eta-2\sigma_j} \leq R^{-g(\eta)}. \]
The scaling of the linear operator~$L$ is then compared with the power nonlinearity~$|\partial_t^\ell u|^p$ in~\eqref{eq:CP}. In particular, for a given scaling parameter~$\eta$, one may prove that global solutions to~\eqref{eq:CP} may exist only if~$p\geq h(\eta)$ (see the proof of Theorem~\ref{thm:main}).

Since our aim is to obtain the largest possible range for nonexistence, we look for parameters~$\bar\eta$ which realize the maximum~$h(\bar\eta)=\max_{\eta\geq0} h(\eta)$ and we apply our test function method with that optimal scaling~$(R^{-\bar\eta}t,R^{-1}x)$. Indeed, as a consequence of the previous discussion, a scaling~$(R^{-\eta}t,R^{-1}x)$ gives the largest possible range for nonexistence if we fix~$\eta=\bar\eta$.

It is clear that $h(\eta)$ admits a maximum value (which may also be~$\infty$ in some cases), since~$h$ is a continuous function on the compact interval~$[0,\infty]$. In this section, we show how to find~$\max_{\eta\in[0,\infty]}h(\eta)$ constructively. In particular, we show that if~$p_c\in(1,\infty)$, then there exists a unique maximum point~$\bar \eta$ in the interval~$(0,\infty)$; moreover, $h$ is increasing in~$[0,\bar\eta]$ and decreasing in~$[\bar\eta,\infty]$.

First of all, we discuss how~$g$ may be conveniently represented.
%]]]

%[[[ \begin{Remark}
\begin{Remark}
\label{Remark:1.9}
The function~$g$ is continuous and piecewise smooth,
	since it is piecewise described by the line~$(j-\ell)\eta + 2\sigma_j$ with some $j$.

Moreover, there exists a unique finite sequence~$\{(j_k,\eta_k)\}_{k=0}^{m_1}$ with~$m_1\leq m$, $\eta_0=0$, and~$\eta_k$ increasing, such that
\begin{equation}\label{eq:gpiece}
\forall \eta \in [\eta_k,\eta_{k+1}]: \quad g(\eta)= (j_k-\ell)\eta + 2\sigma_{j_k},
\end{equation}
where we formally set~$\eta_{m_1+1}=\infty$.

We have the following properties on the sequences~$j_k$ and~$\sigma_{j_k}$:
\begin{itemize}
\item it holds~$j_0=\min\{j: \ a_j\neq0, \sigma_j=0\}$ and~$j_{m_1}=\min\{j: \ a_j\neq0\}$
\item the sequence~$j_k$ is decreasing; this is a consequence of the fact that the slope of the line $(j-\ell)\eta + 2\sigma_j$ is increasing with respect to~$j$;
\item the sequence~$\sigma_{j_k}$ is increasing. Indeed, assume by contradiction that~$\sigma_{j_{k+1}}\leq \sigma_{j_k}$; then, using that~$j_k$ is decreasing, it follows that
\[ (j_k-\ell)\eta + 2\sigma_{j_k} > (j_{k+1}-\ell)\eta + 2\sigma_{j_{k+1}}\geq g(\eta), \quad \forall \eta>0.\]
This gives the contradiction, since $g(\eta)$ could never assume the value~$(j_k-\ell)\eta + 2\sigma_{j_k}$.
\end{itemize}
\end{Remark}
%]]]

Thanks to the representation of the function~$g(\eta)$ given in Remark~\ref{Remark:1.9}, we may easily find a sufficient and necessary condition to get~$p_c=\infty$ in~\eqref{eq:pcritical}.

%[[[ \begin{Remark}\label{Rem:pinfty}
\begin{Remark}\label{Rem:pinfty}
The critical exponent~$p_c$ in~\eqref{eq:pcritical} is~$\infty$
(that is, we have no global-in-time solutions for any power nonlinearity~$p>1$) if, and only if, there exists~$\eta\in[0,\infty)$ such that $g(\eta)-\eta=n$ (since $g(\eta)-\eta$ is a continuous function).

If~$j_{m_1}\geq\ell+2$, we find~$p_c=\infty$, since~$g(\eta)-\eta\nearrow\infty$ as~$\eta\nearrow\infty$.

If~$j_{m_1}=\ell+1$, we find that
\[ h(\infty)= \lim_{\eta\to\infty} \frac{n+\eta}{(n+\eta-g(\eta))_+} = \infty, \]
so that~$p_c=\infty$ as well, due to~$g(\eta)=\eta-2\sigma_{j_m}$ for any~$\eta\in[j_m,\infty]$.

Assume now that~$j_{m_1}\leq \ell$. Thanks to the representation of~$g$ provided by Remark~\ref{Remark:1.9}, the function~$g(\eta)-\eta$ is increasing in~$[0,\eta_{\bar k}]$, and nonincreasing in~$[\eta_{\bar k},\infty]$, where
\[ \bar k = \min \{ k: \ j_k\leq \ell+1\}. \]
As a consequence, $p_c=\infty$ if, and only if,
\[ 2\sigma_{j_{\bar k}} \geq n + (\ell+1-j_{\bar k})\eta_{\bar k}. \]
\end{Remark}
%]]]

Thanks to the representation of the function~$g(\eta)$ given in Remark~\ref{Remark:1.9}, we may easily study the monotonicity of the function~$h(\eta)$ in each interval of type~$[\eta_k,\eta_{k+1}]$.

%[[[ \begin{Remark}\label{rem:mainpc}
\begin{Remark}\label{rem:mainpc}
We assume that~$p_c<\infty$, since the case~$p_c=\infty$ is already discussed in Remark~\ref{Rem:pinfty}. As a consequence, we may replace:
\[ h(\eta)=\frac{n+\eta}{n+\eta-g(\eta)} = 1 + \frac{g(\eta)}{n+\eta-g(\eta)},\]
for any~$\eta\in[0,\infty)$, since the denominator is positive. We first notice that~$h(0)=1$, since~$\sigma_{j_0}=0$, and
\[ h(\infty)=\lim_{\eta\to\infty} \frac{n+\eta}{n+\eta-g(\eta)}= \frac{1}{\ell+1-j_{m_1}} \leq1, \]
due to~$j_{m_1}\leq\ell$.

In order to determine $p_c$, %when $n + \eta > g(\eta)$ for any $\eta \geq 0$,
we may differentiate~$h$ with respect to~$\eta$, for any~$\eta\neq \eta_k$, where $\eta_k$ is given in Remark \ref{Remark:1.9}. We get:
\[ h'(\eta) = \frac{-g(\eta)+(n+\eta)g'(\eta)}{(n+\eta-g(\eta))^2}. \]
The monotone behavior of~$h$ is determined by the sign of~$-g(\eta)+(n+\eta)g'(\eta)$. However, in any interval of type~$(\eta_k,\eta_{k+1})$, the sign is obtained by formula~\eqref{eq:gpiece}, that is,
\[ \forall \eta \in(\eta_k,\eta_{k+1}): \quad -g(\eta)+(n+\eta)g'(\eta) = -2\sigma_{j_k} +n(j_k-\ell). \]
So the sign of~$h'(\eta)$ is constant in every interval~$(\eta_k,\eta_{k+1})$ and is given by
\[ \forall \eta \in(\eta_k,\eta_{k+1}): \quad \sign h'(\eta) = s_k \doteq \sign (-2\sigma_{j_k} +n(j_k-\ell)). \]
Due to the fact that~$j_k$ is decreasing and~$\sigma_{j_k}$ is increasing, we find that~$s_k$ is a decreasing function.

Recalling that~$\sigma_{j_0}=0$, we get~$s_0=1$ if~$j_0\geq\ell+1$ and~$s_0\leq0$ if~$j_0\leq \ell$. In this latter case, $p_c=h(0)= 1$, and we get no result of nonexistence. Therefore, in the following we assume that~$j_0\geq\ell+1$.

On the other hand, $s_{m_1}=-1$ as a consequence of~$j_{m_1}\leq\ell$ (since we assumed~$p_c<\infty$, see Remark~\ref{Rem:pinfty}).

Moreover, we may exclude the case where~$s_k=0$ for some~$k$.
Indeed, if~$s_k=0$, then $2\sigma_{j_k}=n(j_k-\ell)$.
Since $\sigma_{j_k} \geq 0$, the estimate $j_k \geq \ell$ holds.
\eqref{eq:gpiece} implies that we have
\[ g(\eta_{k})= (j_k-\ell)(n+\eta_{k}), \]
so that
\[ h(\eta_{k}) = \frac1{\ell+1-j_k}=\begin{cases}
1 & \text{if~$j_k \geq \ell + 1$,}\\
\infty & \text{if~$j_k=\ell+1$.}
\end{cases} \]
The both cases contradict our assumptions, since $1=h(0)<h(\eta_k)\leq p_c<\infty$, so there is no~$k$ such that~$s_k=0$.

Therefore, there exists a unique~$k$ such that~$s_{k-1}=1$ and~$s_k=-1$; then~$h$ is increasing in~$[0,\eta_k]$ and decreasing in~$[\eta_{k},\infty]$, and the critical exponent is
\[ p_c = h(\eta_k). \]
Moreover, due to~$g(\eta_k)<n+\eta_k$ (since~$p_c<\infty$) and $s_k=-1$, we find the chain of inequalities
\[ n+\eta_k > g(\eta_k) = (j_k-\ell)\eta_k +2\sigma_{j_k} > (n+\eta_k)(j_k-\ell), \]
which gives~$j_k-\ell<1$, that is, $j_k\leq \ell$.
\end{Remark}
%]]]

%[[[ \begin{Remark}
\begin{Remark}
Assume that~$a_0=\ldots=a_{k-1}=0$, for some~$k\geq1$, and~$a_k\neq0$ in~\eqref{eq:L}. As a consequence, $j_{m_1}=k$. We may distinguish two cases. If~$\ell\leq k-1$, then~$p_c=\infty$ due to~$j_{m_1}=k\geq \ell+1$ (see Remark~\ref{Rem:pinfty}). Assume now that~$k\leq \ell$. In this case, we may define~$w=\partial_t^ku$, and reduce the original problem~\eqref{eq:CP} to a problem of order~$m-k$ with power nonlinearity~$|\partial_t^{\ell-k} w|^p$. Indeed, Cauchy problem~\eqref{eq:CP} now reads as
\[
\begin{cases}
\sum_{j=0}^{m-k} a_{j+k} A_{j+k}\partial_t^j w = |\partial_t^{\ell-k} w|^p, & t>0, \ x\in\R^n,\\
\partial_t^j w(0,x)= u_{j+k}(x), & j=0,\ldots,m-1-k.
\end{cases}
\]
For this reason, it is not restrictive to assume~$a_0\neq0$ in meaningful examples, as the ones collected in Section~\ref{sec:examples}.
\end{Remark}
%]]]

%[[[ \begin{Remark}
\begin{Remark}\label{rem:principalpart}
Assume that~$p_c\in(1,\infty)$ and let $\bar\eta$ satisfy $p_c = h(\bar\eta)$.
Then we define
	\[
	J_p = \{ j:\quad a_j\neq0, \ \, g(\bar\eta) = (j-\ell) \bar\eta + 2 \sigma_j \}.
	\]
and say that
	\[
	L_p = \sum_{j \in J_p} \partial_t^j A_j
	\]
is the principal part of $L$. Since there exists~$\eta_k$ such that~$\bar\eta=\eta_k$, it follows that~$j_{k-1},j_k\in J_p$, namely, the principal part of~$L$ contains at least two terms of~$L$: $\partial_t^{j_{k-1}}A_{j_{k-1}}$ and~$\partial_t^{j_{k}}A_{j_{k}}$.

In particular, thanks to Remark \ref{rem:mainpc}, which implies that~$j_k\leq \ell$, at least one index in~$\{0,\ldots,\ell\}$ belongs to~$J_p$. In the special case~$\ell=0$, this means that~$A_0$ belongs to the principal part~$L_p$ of~$L$.

We may say that~$L_p$ is a quasi-homogeneous operator (of type~$(g(\bar\eta)+\ell\bar\eta,\bar\eta,1)$), in analogy to Definition~2.2 in~\cite{DAL03} (see also~\cite{Gru,Matsuzawa}). Indeed, our critical exponent~$p_c$ is consistent with the one defined in~\cite{DAL03}.
\end{Remark}
%]]]
%]]]

%[[[ \section{Proof of Proposition %ref{Propositoin:1}%}
%\section{Proof of Theorems \ref{thm:main} and~\ref{thm:main2}}\label{sec:proofs}
\section{Proof of Theorem~\ref{thm:main}}\label{sec:proof}
%[[[ We assume that $p \leq p_c $ and we show that if $u$ is a global-in-time weak solution to~\eqref{eq:CP},
We assume that $p \leq p_c $ and we show that if $u$ is a global-in-time weak solution to~\eqref{eq:CP},
according to Definition~\ref{def:weak}, then~$u\equiv0$.
Due to the initial conditions, it is sufficient to show that the function~$v=\partial_t^\ell u$ is identically zero. Indeed, if this is true, then
\[ u(t,\cdot) =c_0 + c_1 t + \ldots + c_{\ell-1} t^{\ell-1} \]
for any~$t\geq0$, for some~$c_0,\ldots,c_{\ell-1}\in\R$. Imposing the initial conditions, we derive $c_0=\ldots=c_{\ell-1}=0$, so that~$u$ is identically zero.

Let us prove that~$v \equiv 0$. Recalling that~$u$ is a global-in-time weak solution to~\eqref{eq:CP}, according to Definition~\ref{def:weak}, we fix suitable test functions~$\psi$ and~$\varphi$, depending on a parameter~$R\gg1$, on which we test the integral equality in~\eqref{eq:1.4}.

Let $\chi$ be a smooth decreasing function satisfying $\chi(t)=1$ for any $0 \leq t \leq 1/2$ and $\chi(1)=0$, and fix~$\psi(t)=(\chi(t))^{mp'}$. Let $\eta \geq 0$. For any~$R\gg1$, we define
\[ \psi_R (t) = \psi (R^{-\eta}t), \qquad \varphi_R(x)=\<R^{-1}x\>^{-q}.\]
We remark that~$\varphi_R\in\mathcal C_q^\infty$, thanks to Corollary~\ref{cor:test}. We also preliminarily notice that
\[ (A_j\varphi_R)(x) = R^{-2\sigma_j}\,(A_j\varphi)(R^{-1}x).\]
Recalling that~$\partial_t^\ell u=v$, the integral equality in~\eqref{eq:1.4} reads as
\begin{align*}
&\int_0^\infty \psi_R(t)\,\int_{\R^n}|v(t,x)|^p\,\varphi_R(x)\,dx\,dt\\
& \qquad = \sum_{j=0}^{m} (-1)^{(j-\ell)} \int_0^\infty \psi_R^{(j-\ell)}(t)\,\int_{\R^n} v(t,x)\,A_j\varphi_R(x)\,dx\,dt\\
& \qquad \qquad - \sum_{j=\ell}^{m-1} \int_{\R^n}u_{j}(x)\,A_{j+1}\varphi_R(x)\,dx.
\end{align*}
At first, we obtain that the identity
	\begin{align*}
	&\lim_{R \to \infty}
	\sum_{j=\ell}^{m-1} \int_{\mathbb R^n} u_{j}(x) (A_{j+1} \varphi_R) (x) dx\\
	&=
	\sum_{j=\ell}^{m-1} \lim_{R \to \infty}
	R^{-2\sigma_{j+1}} \int_{\mathbb R^n} u_{j}(x) (A_{j+1} \varphi) (R^{-1} x) dx\\
	&=
	\sum_{j \in I} a_j \int_{\mathbb R^n} u_{j}(x)dx
	\end{align*}
follows from the Lebesgue dominant convergence theorem, due to~$u_j\in L^1(\<x\>^{q}dx)$ and~$A_{j+1}\varphi\in L^\infty(\<x\>^{-q}dx)$. In the last equality, we used that~$A_{j+1}\varphi =a_{j+1}\varphi$, if~$j\in I$. Due to the sign assumption~\eqref{eq:sign}, the latter term in the identity is strictly positive.

As a consequence of the previous identity and of the sign assumption~\eqref{eq:sign}, we obtain the inequality
	\begin{align}
	&\limsup_{R\to \infty} \sum_{j=0}^{m} \int_0^\infty \int_{\R^n}
	v(t,x)\,(-1)^{j-\ell} \psi_R^{(j-\ell)}(t) (A_{j} \varphi_R) (x) \,dx\,dt
	\nonumber\\
	&> \limsup_{R\to \infty} \int_0^\infty \int_{\R^n}|v(t,x)|^p \,\psi_R(t)\,\varphi_R(x)\,dx\,dt
    \nonumber\\
	\label{eq:3.2}
    & = \int_0^\infty \int_{\R^n}|v(t,x)|^p\,dx\,dt.
	\end{align}
The latter equality in~\eqref{eq:3.2} has to be understood in the sense of an integral of a nonnegative function, which may be a nonnegative number or~$\infty$. Indeed, by Beppo-Levi monotone convergence theorem for nonnegative increasing sequences, we find the limit as a consequence of $\psi_R(t)\,\varphi_R(x)\nearrow 1$ as~$R\nearrow\infty$.

Corollary~\ref{cor:test} implies that
	\begin{align}
	|(A_{j} \varphi_R) (x)|=R^{-2\sigma_j}|(A_j\varphi)(R^{-1}x)|
	\leq C R^{-2 \sigma_j} \varphi(R^{-1}x)=C R^{-2 \sigma_j}\varphi_R(x).
	\label{eee}
	\end{align}
We also assert that
	\begin{align}
	| \psi_R^{(j-\ell)} (t)|
	\leq C R^{-(j-\ell) \eta} \psi_R(t)^{1/p}.
	\label{aaa}
	\end{align}
Indeed, when $j-\ell \geq 0$, \eqref{aaa} is directly computed:
\begin{align*}
| \psi_R^{(j-\ell)} (t)|
    & =R^{-(j-\ell) \eta} | (\chi^{mp'})^{(j-\ell)}|(R^{-\eta}t) \\
    & \leq C R^{-(j-\ell) \eta} \chi^{m(p'-1)}(R^{-\eta}t)\\
    & = C R^{-(j-\ell) \eta} \psi_R(t)^{1/p}.
\end{align*}
When $j-\ell < 0$, since
	\[
	\psi_R^{(-1)}(t)
	= - \int_t^\infty \psi (R^{-\eta} \tau ) d \tau
	= R^{\eta} \psi^{(-1)}(R^{-\eta} t),
	\]
the identity
	\begin{align}
	\psi_R^{(j-\ell)}(t)
	= R^{ - (j-\ell) \eta} \psi^{(j-\ell)}(R^{-\eta} t),
	\label{bbb}
	\end{align}
is shown inductively.
Since for any~$t\in[0,1]$, we have
	\[
	|\psi^{(-1)}(t)|
	= \bigg| \int_t^\infty \psi (\tau ) d \tau \bigg|
	\leq \psi (t ) \bigg| \int_t^1 d \tau \bigg|
	\leq \psi (t )
	\]
because $\psi (t) = 0$ for any $t \geq 1$,
the estimate
	\begin{align}
	|\psi^{(j-\ell)}(t)|
	\leq \psi (t )
	\label{ccc}
	\end{align}
holds if $j-\ell < 0$. Estimates \eqref{bbb} and \eqref{ccc} imply that \eqref{aaa} holds also when $j - \ell < 0$.

Then \eqref{eee}, \eqref{aaa}, and the H\"older inequality imply that we have
	\begin{align}
	&\bigg| \sum_{j=0}^{m} (-1)^{j-\ell} \int_0^\infty \int_{\R^n}v(t,x)
	\, \psi_R^{(j-\ell)}(t)\, (A_{j} \varphi_R) (x) \, dx\,dt \bigg|
	\nonumber \\
	&\leq C R^{-(j-\ell) \eta - 2 \sigma_j + (n+\eta)/p'}
	\bigg( \int_0^\infty \int_{\R^n}|v(t,x)|^p \,\psi_R(t)\varphi_R(x)\,dx\,dt \bigg)^{1/p}.
	\label{eq:3.3}
	\end{align}
Here we just used that
\[ \Big(\int_0^{R^\eta} \int_{\R^n} \<R^{-1}x\>^{-qp'} \,dx \,dt \Big)^{p'}
= C\,R^{\frac{\eta+n}{p'}}. \]
Recalling the definition of~$g(\eta)$ and~$h(\eta)$ in Definition~\ref{def:homdim}, we get
	\begin{align*}
	-( (j-\ell) \eta + 2 \sigma_j ) p' + (n+\eta)
	&\leq - g(\eta) \frac{p}{p-1} + (n+\eta)\\
	&\leq \frac{1}{p-1} \bigg( (n+\eta- g(\eta)) p - (n+\eta) \bigg)\\
	&\begin{cases}
	< 0
	&\mathrm{if} \quad n+\eta- g(\eta) \leq 0,\\
	= \frac{(n+\eta- g(\eta)}{p-1} (p - h(\eta))
	&\mathrm{if} \quad n+\eta- g(\eta) > 0.
	\end{cases}
	\end{align*}
Therefore, if~$p<p_c=\sup_{\eta\geq0}h(\eta)$, there exists some $\eta \geq 0$ such that
	\[
	-( (j-\ell) \eta + 2 \sigma_j ) p' + (n+\eta)
	< 0.
	\]
This inequality, \eqref{eq:3.2}, \eqref{eq:3.3}, and the monotone convergence theorem imply that
	\[
	\int_0^\infty \int_{\R^n}|v(t,x)|^p \,dx\,dt
	= \lim_{R \to 0} \int_0^\infty \int_{\R^n}|v(t,x)|^p \,\psi_R(t) \phi_R(x) \,dx\,dt
	\to 0.
	\]
As a consequence, $v\equiv0$, which is impossible, as a consequence of the sign condition~\eqref{eq:sign}, which implies non-trivial data.

%Therefore, if $u$ is a global-in-time weak solution to \eqref{eq:CP}, then $p \geq p_c$.

Now let~$p=p_c$. By~\eqref{eq:3.2} we derive
\begin{align*}
& \int_0^\infty \int_{\R^n}|v(t,x)|^{p_c}\,dx\,dt \\
& \qquad < \limsup_{R\to \infty} \sum_{j=0}^{m} \int_0^\infty \int_{\R^n} v(t,x)\,(-1)^{j-\ell} \psi_R^{(j-\ell)}(t) (A_{j} \varphi_R) (x) \,dx\,dt <\infty,
\end{align*}
that is, $v\in L^{p_c}([0,\infty)\times\R^n)$.

We now repeat the reasoning for the subcritical case~$p\in(1,p_c)$, but we replace the test function~$\varphi_R$ by the test function~$\varphi_{RK}=\<R^{-1}K^{-1}x\>^{-q}$, for a given constant~$K\gg1$.

By the fact that~$v\in L^{p_c}([0,\infty)\times\R^n)$, the dominated convergence theorem gives us
\[ \int_0^\infty \int_{\R^n}|v(t,x)|^{p_c} \,dx\,dt = \lim_{R \to \infty} \int_0^\infty \int_{\R^n}|v(t,x)|^{p_c} \,\psi_{R}(t)\,\varphi_{RK}(x)\,dx\,dt.\]
On the other hand, we notice that
\[ (\partial_t^{j-\ell}\psi)(R^{-\eta}t)\to 0, \ \text{for any~$j=\ell+1,\ldots,m-1$,} \]
pointwisely, as~$R\to\infty$. As a consequence, if we define
\[ I_j(R)=\int_0^\infty \int_{\R^n}|v(t,x)|^{p_c}
	\frac{|\psi^{(j-\ell)}(R^{-1}t) (A_{j} \<\cdot\>^{-q}) (R^{-1}K^{-1}x)|^{p_c}}{\psi_R(t)\varphi_{RK}(x)} dx \, dt, \]
then
\[ \lim_{R\to\infty} I_j(R) = \begin{cases}
0 & \text{if~$j\geq\ell+1$,} \\
B_j & \text{if~$j=0,\ldots,\ell$,}
\end{cases} \]
where
\[ B_j=|\psi^{(j-\ell)}(0) (A_{j} \<\cdot\>^{-q}) (0)|^{p_c}\,\int_0^\infty \int_{\R^n}|v(t,x)|^{p_c} \,dx\,dt, \]
is independent of~$K$ (we notice that~$|\psi^{(j-\ell)}(0)|\leq1$, whereas the explicit value of~$(A_{j} \<\cdot\>^{-q}) (0)$ is computed in \eqref{eq:der0}).
Since~$u$ is a weak solution, it holds
\begin{align*}
& \int_0^\infty \int_{\R^n}|v(t,x)|^{p_c} \,dx\,dt + \sum_{j \in I} a_j \int_{\mathbb R^n} u_{j}(x)dx \\
	&\qquad \leq \limsup_{R\to \infty} \sum_{j=0}^{m} \int_0^\infty \int_{\R^n}
	v(t,x)\,(-1)^{j-\ell} \psi_R^{(j-\ell)}(t) (A_{j} \varphi_{RK}) (x) \,dx\,dt.
\end{align*}
By the sign assumption~\eqref{eq:sign} and using~$p=p_c$, we deduce a contradiction, since
\begin{align*}
0   & < \sum_{j \in I} a_j \int_{\mathbb R^n} u_{j}(x)dx \\
    & \leq \int_0^\infty \int_{\R^n}|v(t,x)|^{p_c} \,dx\,dt + \sum_{j \in I} a_j \int_{\mathbb R^n} u_{j}(x)dx\\
	&\qquad \leq C \sum_{j=0}^{m} K^{-2\sigma_j+\frac{n}{p'}}\, \limsup_{R\to \infty} I_j(R)^{\frac1p} \\
    & \qquad = C \sum_{j=0}^{\ell} K^{-2\sigma_j+\frac{n}{p'}}\, B_j^{\frac1p} \leq C_1\,\sum_{j=0}^\ell K^{-2\sigma_j+\frac{n}{p'}},
\end{align*}
and the quantity~$K^{-2\sigma_j+\frac{n}{p'}}$ is arbitrarily small for any~$j=0,\ldots,\ell$, for a sufficient large~$K\gg1$. Indeed, the sign property~$-2\sigma_j+n/p'<0$ is a consequence of the fact that, for any~$j\leq\ell$, it holds
\[ -2\sigma_j+\frac{n}{p'} < -2\sigma_j+\frac{n}{p'} + \eta (\ell-j+1/p') \leq 0. \]
This concludes the proof.
%]]]

%[[[ \begin{Remark}
\begin{Remark}
The history of the approach with the second parameter $K$ in the critical case when $p=p_c$
can go back at least to \cite{FK10}.
\end{Remark}
%]]]

%[[[ \begin{Remark}
\begin{Remark}
In the critical case where $p=p_c$,
an alternative proof of Theorem~\ref{thm:main}, which does not rely on the second parameter~$K$, may work
under a certain condition.
Indeed,
when $\sigma_j$ is integer for any $j \geq 0$,
one may deploy the classical test function method, which does not rely on the second parameter,
with a test function $\Phi$ satisfying that $\partial_t^j A_j \Phi(0,0) = 0$ for any $j \geq 0$.
This approach does not seem applicable to our case in general
because such a positive function is not known when some $\sigma_j$ are not integer.
However
\eqref{eq:der0} implies that
for any $q > n$ and $\sigma, \varepsilon > 0$,
there exists $\theta_{\sigma,q,\varepsilon} \in (0,1)$ satisfying
	\[
	(-\Delta)^{\sigma} ( \langle \cdot \rangle^{-q} - \theta_{\sigma,q,\varepsilon} \langle \cdot \rangle^{-q-\varepsilon})(0) = 0.
	\]
Therefore,
the classical test function method may work
with
	\[
	\varphi(x) = \langle x \rangle^{-q} - \theta_{\sigma,q,\varepsilon} \langle x \rangle^{-q-\varepsilon}
	\]
provided that the main part~$L_p$ of the operator~$L$, defined in Remark~\ref{rem:principalpart},
contains only one term~$A_{\bar \jmath}$, with~$\bar \jmath\leq\ell$ and $\sigma_{\bar \jmath}$ is not integer.
Namely, assume that~$J_p$ contains exactly one index~$\bar \jmath$. We mention that this is always the case if~$\ell=0$.
\end{Remark}
%]]]
%]]]

%[[[ \section{Examples}\label{sec:examples}
\section{Examples}\label{sec:examples}

Here we present examples of equations for which we explicitly compute the critical exponent provided by Theorem~\ref{thm:main}, also employing the properties discussed in Section~\ref{sec:critical}. Every time we say that~$p_c$ is the critical exponent, we imply that Theorem~\ref{thm:main} may be used to prove nonexistence of global weak solutions for~$1<p\leq p_c$. %, and possibly Theorem~\ref{thm:main2} may be used to extend the nonexistence result to the critical case~$p=p_c$.

%[[[
\begin{Example}\label{ex:heat}
We consider the fractional heat equation,
\[
Lu = u_{t} + (-\Delta)^\sigma u = |u|^p,
\]
that is, $m=1$, $\ell=0$, and~$\sigma_0=\sigma>0$. Then
\[ g(\eta)=\min\{\eta,2\sigma\}=\begin{cases}
\eta & \text{if~$\eta\in[0,2\sigma]$,}\\
2\sigma & \text{if~$\eta\in[2\sigma,\infty]$.}
\end{cases} \]
As a consequence, $\eta_1 = 2\sigma$ and
\[ p_c  = h(2 \sigma) = 1 + \frac{2\sigma}{n}.\]
In particular, if $\sigma=1$, $p_c$ coincides with the well-known Fujita exponent $1 + \frac 2 n$ (\cite{F66, H73, KST77}).
By standard methods,
it is easy to show that global existence of small data solutions holds for supercritical powers~$p>p_c$.
For details, see \cite{FK10, S75} and reference therein.
\end{Example}
%]]]

%[[[
\begin{Example}\label{ex:Ebert}
We consider the fractional $\sigma$-evolution equation of second order with power nonlinearity $|u|^p$,
\[
Lu = u_{tt} + (-\Delta)^\sigma u = |u|^p,
\]
that is, $m=2$, $\ell=0$, $a_1=0$ and $\sigma_0=\sigma>0$. Then
\[ g(\eta)=\min\{2\eta,2\sigma\}=\begin{cases}
2\eta & \text{if~$\eta\in[0,\sigma]$,}\\
2\sigma & \text{if~$\eta\in[\sigma,\infty]$.}
\end{cases} \]
As a consequence, $\eta_1 = \sigma$ and
\[ p_c  = h(\sigma) = 1 + \frac{2\sigma}{n-\sigma},\]
provided that~$\sigma<n$. For~$\sigma>1$, the global existence of small data solutions for supercritical powers~$p>p_c$ in low space dimension has been recently proved in~\cite{EL}.
\end{Example}
%]]]

%[[[ \begin{Example}
\begin{Example}\label{ex:damp}
We consider a damped fractional $\sigma$-evolution equation of second order with power nonlinearity $|u|^p$,
\[ Lu = u_{tt} + (-\Delta)^{\sigma_1} u_t + (-\Delta)^{\sigma} u = |u|^p, \]
where~$\sigma_0=\sigma>0$ and~$\sigma_1\geq0$. Then $m=2$ and $\ell=0$. We compute
\[ g(\eta) = \min \{ 2\eta, \ \eta+2\sigma_1, 2\sigma\}. \]
We shall distinguish three cases.

\bigskip

If~$\sigma_1=0$, the damping is called classical, exterior, or weak, and
\[ g(\eta) = \begin{cases}
\eta & \text{if~$\eta\in[0,2\sigma]$,}\\
2\sigma & \text{if~$\eta\in[2\sigma,\infty]$.}
\end{cases} \]
As a consequence, $\eta_1=2\sigma$ and
\[ p_c  = h(2 \sigma) = 1 + \frac{2\sigma}{n}, \]
as in Example~\ref{ex:heat}. Indeed, the principal part of the operator is~$L_p=\partial_t+(-\Delta)^\sigma$, the fractional heat operator. The diffusion phenomenon for this model, that is, the solution to~$Lu=0$ asymptotically behaves as the solution to~$L_pu=0$ for a suitable choice of initial data, has been investigated in~\cite{K00} for fractional powers (for the integer case, we refer to~\cite{HM, HT, MN03, N03}).

\bigskip

If~$0<\sigma_1<\sigma/2$, the damping is called structural and effective, and
\[ g(\eta)=\begin{cases}
2\eta & \text{if~$\eta< 2\sigma_1$,}\\
\eta+2\sigma_1 & \text{if~$2\sigma_1<\eta< 2(\sigma-\sigma_1)$}\\
2\sigma & \text{if~$\eta>2(\sigma-\sigma_1)$.}
\end{cases} \]
Then $\eta_1=2\sigma_1$, $\eta_2=2(\sigma-\sigma_1)$, and the critical exponent is
\[ p_c = h(2(\sigma-\sigma_1)) = 1 + \frac{2\sigma}{n-2\sigma_1}, \]
provided that~$2\sigma_1<n$. The global existence of small data solutions for supercritical powers~$p>p_c$, in low space dimension, has been proved in a series of papers~\cite{DAE14NA, DAE17, DAR14}. The principal part of the operator is~$L_p=(-\Delta)^{\sigma_1}\partial_t+(-\Delta)^\sigma$. Indeed, a diffusion phenomenon also holds for this model~\cite{DAE14JDE}, similarly to the case~$\sigma_1=0$.

\bigskip

If~$2\sigma_1>\sigma$, the damping is called structural and noneffective, and
\[ g(\eta)=\begin{cases}
2\eta & \text{if~$\eta< \sigma$,}\\
2\sigma & \text{if~$\eta>\sigma$.}
\end{cases} \]
In such a case, $\eta_1=\sigma$ and
\[ p_c = h(\sigma)=1+\frac{2\sigma}{n-\sigma}, \]
provided that~$\sigma<n$, as in Example~\ref{ex:Ebert}.  The global existence of small data solutions for supercritical powers~$p>p_c$, in low space dimension, is proved in~\cite{DAE20}.

The adjective ``noneffective'' for the damping hints to the fact that the principal part of the operator, $L_p=\partial_{tt}+(-\Delta)^\sigma$, does not contain the damping (see also the classification introduced in~\cite{DAE16}).

In the limit case~$2\sigma_1=\sigma$, the critical exponent is the same as in the effective and noneffective case, but the operator~$L$ is quasi-homogeneous, that is, $L_p=L$.
\end{Example}
%]]]

%[[[ \begin{Example}
\begin{Example}\label{ex:dampt}
We consider a damped wave equation, as in Example~\ref{ex:damp}, but with nonlinearity $|u_t|^p$, that is,
\[ Lu= u_{tt} + (-\Delta)^{\sigma_1} u_t + (-\Delta)^\sigma u = |u_t|^p. \]
Then $m=2$, $\ell=1$, $\sigma_1\geq0$ and~$\sigma_0=\sigma>0$. We compute
\[ g(\eta) = \min \{ \eta, \ 2\sigma_1, -\eta+2\sigma\}. \]
If~$\sigma_1=0$ (classical damping), then~$j_0=1$ and~$\eta_1=2\sigma$. Due to~$g(\eta_1)=0$, we do not have a nonexistence result. So, let~$\sigma_1>0$, that is, we consider a structural damping. We distinguish two cases.

\bigskip

If the damping is effective, that is, $2\sigma_1<\sigma$, then
\[ g(\eta)=\begin{cases}
\eta & \text{if~$\eta< 2\sigma_1$,}\\
2\sigma_1 & \text{if~$2\sigma_1<\eta, 2(\sigma-\sigma_1)$}\\
-\eta+2\sigma & \text{if~$\eta>2(\sigma-\sigma_1)$.}
\end{cases} \]
As a consequence, $\eta_1=2\sigma_1$, $\eta_2=2(\sigma-\sigma_1)$, and
\[ s_0=1, \quad s_1=s_2=-1, \]
so that
\[ p_c = h(2\sigma_1)=1+\frac{2\sigma_1}{n}. \]
The global existence of small data solutions for supercritical powers~$p>p_c$, in low space dimension, is proved in~\cite{DAE17}. The principal part of the operator is~$L_p=\partial_t^2+(-\Delta)^{\sigma_1}\partial_t$. Indeed, a diffusion phenomenon also holds for this model~\cite{DAE14JDE}. Indeed, the problem for~$L_pu=|u_t|^p$ may be reduced to the problem for the fractional heat equation~$\partial_tv+(-\Delta)^{\sigma_1}v=|v|^p$, treated in Example~\ref{ex:heat}, setting~$v=u_t$. Therefore, it is a natural outcome that the critical exponent is the same for the two problems.

%We also mention that, following the ideas in the proof of Theorem~\ref{thm:main2}, one may prove the nonexistence of global-in-time solutions also in the critical case~$p=1+2\sigma_1/n$, if~$\sigma_1\in(0,3)$, thanks to the property that~$A_0=(-\Delta)^\sigma$ is not in the principal part of the operator in the effective case.

\bigskip

If the damping is noneffective, that is, $2\sigma_1>\sigma$, then
\[ g(\eta)=\begin{cases}
\eta & \text{if~$\eta< \sigma$,}\\
-\eta+2\sigma & \text{if~$\eta>\sigma$.}
\end{cases} \]
In such a case, $\eta_1=\sigma$ is the best scaling and
\[ p_c = h(\sigma)=1+\frac{\sigma}{n}. \]
The global existence of small data solutions for supercritical powers~$p>p_c$, in low space dimension, is proved in~\cite{DAE20}. The principal part of the operator is~$L_p=\partial_t^2+(-\Delta)^{\sigma}$ and it does not contain the damping term.
\end{Example}
%]]]

%[[[ \begin{Example}
\begin{Example}
Let us consider an operator~$L$ as in~\eqref{eq:L} and assume that~$p_c=h(\eta)$ for some~$\eta\in(0,\infty)$. Then its principal part~$L_p$ is a quasi-homogeneous operator of order~$m_p\leq m$, in the sense that:
\begin{equation}\label{eq:quasihom}
L_pu= \sum_{j=0}^{m_p} b_j (-\Delta)^{\sigma_{m_p} + (m_p-j)\theta}\partial_t^j u,
\end{equation}
where~$\sigma_{m_p}\geq0$ and $\theta=1/(2\eta)$. Here~$b_j\in\R$, with~$b_{m_p}\neq0$ and at least one among~$b_0,\ldots,b_{m_p-1}$ is nonzero.

\bigskip

On the other hand, if~$L=L_p$ is a quasi-homogeneous operator in the form~\eqref{eq:quasihom}, for some~$\theta>0$, and we consider problem~\eqref{eq:CP}, then
\[ g(\eta)=2\sigma_{m_p}+(m_p-\ell)\,\min\{\eta,2\theta\}.\]
In particular, setting~$\sigma=\sigma_\ell=\sigma_{m_p}+(m_p-\ell)\theta$, we obtain that
\[ p_c = h(2\theta) = 1 + \frac{2\sigma}{n+2\theta -2\sigma}, \]
provided that~$2\sigma<n+2\theta$.
\end{Example}
%]]]
%]]]

%[[[ \section{The case of integer powers}\label{sec:integer}
\section{The case of integer powers}\label{sec:integer}

For the ease of reading, we collect in this Section the basics of the classical test function method, applied in the case in which the critical exponent is determined via the analogous of Definition~\ref{def:homdim}, in space dimension~$n=1$. The reason to fix space dimension~$n=1$ is that in higher space dimension, a differential operator could be not homogeneous in space, in general, so the critical exponent requires different calculations to be computed in space dimension~$n\geq2$. However, if all the derivatives~$\partial_x^{r_j}$ are replaced by~$(-\Delta)^{\frac{r_j}2}$, with even integer~$r_j$, in~\eqref{eq:Lint}, then our result remains valid also in higher space dimension~$n\geq2$.

\bigskip

Let~$L$ be an operator of order~$m$ in the time variable,
\begin{equation}\label{eq:Lint}
L= \partial_t^m + \sum_{j=0}^{m-1} a_j \partial_x^{r_j} \partial_t^{j},
\end{equation}
where~$r_j$ is an integer number. Consistently with Definition~\ref{def:homdim}, we now put
\[
g(\eta)= \min_{j=0,\ldots,m, a_j \neq 0} \{ (j-\ell) \eta + r_j \}.
\]
for any~$\eta\in[0,\infty]$, and we define~$p_c$ as in~\eqref{eq:pcritical}, the critical exponent for~\eqref{eq:CP}. Due to the fact that now the derivatives are integer, we may rely on a more classical definition of weak solution.
\begin{Definition}\label{def:weakint}
Let~$L$ be as in~\eqref{eq:Lint}. Assume that the initial data in~\eqref{eq:CP} verify the assumption
\[ \text{$u_j=0$ if~$j=0,\ldots,\ell-1$ and~$u_j\in L^1_\lloc(\R)$ if $j \geq \ell$.}\]
We fix~$T\in(0,\infty]$. We say that~$u \in W^{\ell,p}_\lloc\big([0,T),L^p_\lloc(\R))\big)$ is a weak solution to~\eqref{eq:CP} if~$\partial_t^ju(0,\cdot)=0$ for any~$j\leq\ell-1$, and for any function~$\psi\in\mathcal C_c^\infty([0,T))$, with~$\psi=1$ in a neighborhood of~$0$ and for any function~$\varphi\in \mathcal C_c^\infty(\R)$, it holds
	\begin{align*}
	&\int_0^T \psi(t)\,\int_{\R}|\partial_t^\ell u(t,x)|^p\,\varphi(x)\,dx\,dt\\
	& \qquad = \sum_{j=0}^{m} (-1)^{(j+r_j-\ell)} \int_0^T \psi^{(j-\ell)}(t)\,\int_{\R} \partial_t^\ell u (t,x)\,\partial_x^{r_j}\varphi(x)\,dx\,dt\\
	& \qquad \qquad - \sum_{j=\ell}^{m-1} (-1)^{r_{j+1}} \int_{\R}u_{j}(x)\,\partial_x^{r_{j+1}}\varphi(x)\,dx,
	\end{align*}
where for $j < 0$, $\psi^{(j)}$ is the compactly supported primitive of~$\psi^{(j+1)}$, %inductively given by
	\[
	\psi^{(j)}(t) 	= - \int_t^T \psi^{(j+1)}(\tau) d\tau.
	\]
\end{Definition}
It is easy to show that classical solutions~$u\in\mathcal C^\infty([0,\infty)\times\R)$ are weak solutions, integrating by parts. Then we have the following.
%
%{[[ \begin{Theorem}\label{thm:intmain}
\begin{Theorem}\label{thm:intmain}
Let~$L$ be as in~\eqref{eq:Lint}, and~$p_c$ be as defined above. We define
\[ I = \{ j \geq \ell: \ r_{j+1}=0,\ a_{j+1} \neq 0\}. \]
We assume that~$u_j=0$ for any~$j\leq \ell-1$, that~$u_j\in L^1_\lloc(\R)$, for any~$j=\ell,\ldots,m-1$, with~$j\not\in I$, and that~$u_j\in L^1(\R)$ for any~$j\in I$. Moreover, we assume the sign condition~\eqref{eq:sign}. If there exists a global-in-time weak solution~$u \in W^{\ell,p}_\lloc\big([0,\infty),L^p_\lloc(\R)\big)$ to~\eqref{eq:CP}, according to Definition~\ref{def:weakint}, then $p>p_c$.
\end{Theorem}
\begin{proof}
The proof of Theorem~\ref{thm:intmain} is a classical application of the test function method, in particular it is completely analogous to the proof of Theorem~\ref{thm:main}, but we now fix~$\varphi(x)=\psi(|x|)$ as a test function in space, where the test function in time~$\psi(t)$ is defined as in the proof of Theorem~\ref{thm:main}.

However, due to the fact that the derivatives are integer, in the critical case~$p=p_c$,
	we may directly deduce that~$v=\partial_t^\ell u$ is identically zero, with no need to use the parameter~$K\gg1$ as in the proof of Theorem~\ref{thm:main}. %as we did in the proof of Theorem~\ref{thm:main2}.
Namely, thanks to
\[ (\partial_t^j\psi)(R^{-\eta}t)\to 0, \ \text{for any~$j=1,\ldots,m-1$, and} \quad \partial_x^{r_0}(R^{-1}x)\to (\partial_x^{r_0}\varphi)(0)=0, \]
pointwisely, as~$R\to\infty$, applying dominated convergence theorem and H\"older inequality.
\end{proof}
%]]]

%[[[ \section*{Acknowledgments}
\section*{Acknowledgments}
The first author is supported by Grants-in-Aid for JSPS Fellows 19J00334
and Early-Career Scientists 20K14337.
The authors thank Prof. Lorenzo D'Ambrosio for the kind discussion and his useful suggestions.
%]]]

%[[[ \begin{thebibliography}{00}

%]]]

\end{document}